\documentclass[11pt]{article}

\usepackage[margin=1in]{geometry}
\usepackage{amsmath,amssymb,amsthm}
\usepackage{mathtools}
\usepackage[hidelinks]{hyperref}

\newtheorem{theorem}{Theorem}
\newtheorem{lemma}[theorem]{Lemma}
\newtheorem{proposition}[theorem]{Proposition}
\newtheorem{corollary}[theorem]{Corollary}

\newcommand{\chips}[1]{\lvert #1\rvert}
\newcommand{\activity}{A}

\title{The Middle Stair for Complete Bipartite Parallel Chip-Firing}
\author{Minkyu Jung\\
\small\texttt{minkyu.jung3@gmail.com}}
\date{August 2026}

\begin{document}
\maketitle

\begin{abstract}
We prove the middle-stair conjecture for every complete bipartite graph.
If a parallel chip-firing game on \(K_{a,b}\) has configuration \(\sigma\)
with
\[
2ab-a-b<\chips{\sigma}<2ab,
\]
then its eventual period is \(2\).  The balanced case \(K_{a,a}\) was
proved by Ji, Li, and Wang using one-parameter conjugate configurations.
We introduce two-parameter conjugates \(c^{k,\ell}\), in which the rank
shift on one side supplies the additive offset on the other.  These
conjugates preserve both the total number of chips and the activity.
An exact Ferrers-diagram count then produces a nonnegative conjugate with
two-round firing coverage on one side.  The coverage propagates in
alternating two-round waves, giving activity \(1/2\); non-clumpiness then
forces period \(2\).
\end{abstract}



\section{Introduction}

In the parallel chip-firing game on a finite graph, every vertex with at
least as many chips as its degree simultaneously sends one chip to each
neighbor.  Bitar and Goles introduced the model and observed its eventual
periodicity \cite{bitar-goles}.  Levine showed that activity on complete
graphs exhibits a devil's staircase \cite{levine}.  The middle stair leads
to the conjecture that
\[
2|E|-|V|<\chips{\sigma}<2|E|
\]
forces period \(2\).

Ji, Li, and Wang proved the conjecture for \(K_{a,a}\), as well as excluding
periods \(3\) and \(4\) on general graphs \cite{ji-li-wang}.  Jiang had
previously determined all possible period lengths on \(K_{a,b}\), without
classifying them by chip density \cite{jiang}.  We settle the
density-conditioned question for all \(K_{a,b}\).

\begin{theorem}\label{thm:main}
Let \(a,b\) be positive integers, and let \(\sigma\) be a configuration
for parallel chip-firing on \(K_{a,b}\).
If
\[
2ab-a-b<\chips{\sigma}<2ab,
\]
then the eventual period of \(\sigma\) is \(2\).
\end{theorem}

The main device is a two-parameter version of the conjugates used in the
balanced case.  The two parameters are not cosmetic: when \(a\ne b\), they
are what makes both side-total changes cancel.

\section{Preliminaries}

Let \(L=\{L_1,\ldots,L_a\}\) and \(R=\{R_1,\ldots,R_b\}\) be the two sides
of \(K_{a,b}\).  Thus vertices in \(L\) have degree \(b\), while vertices
in \(R\) have degree \(a\).  Write \(U\sigma\) for the result of one
parallel update, and let \(F_t(v)\in\{0,1\}\) indicate whether \(v\) fires
from \(U^t\sigma\).  Put
\[
u_t(\sigma,v)=\sum_{s=0}^{t-1}F_s(v).
\]
For a configuration \(\rho\), let \(r_L(\rho)\) and \(r_R(\rho)\)
denote the numbers of firing vertices in \(L\) and \(R\), respectively.
We write
\[
\chips{\rho_L}=\sum_{v\in L}\rho(v),\qquad
\chips{\rho_R}=\sum_{w\in R}\rho(w).
\]
The activity is the limiting average fraction of firing vertices,
\[
\activity(\sigma)=
\lim_{t\to\infty}\frac{1}{(a+b)t}\sum_{v\in L\cup R}u_t(\sigma,v).
\]

We use three standard facts.  On a periodic orbit of a connected graph,
every vertex fires the same number of times per period
\cite[Proposition~2.5]{jiang}.  A nontrivial periodic configuration
satisfies
\[
0\le \sigma(v)\le 2\deg(v)-1.
\tag{2.1}\label{eq:confined}
\]
This follows from the eventual upper bound and confinement lemmas
\cite[Lemmas~2.1 and~2.2]{jiang}.  Finally, a periodic firing word cannot
contain both \(00\) and \(11\)
\cite[Theorem~6.2]{scully-jiang-zhang}.

\begin{lemma}[Within-side spread]\label{lem:spread}
On a nontrivial periodic orbit,
\[
\max_{v\in L}\sigma(v)-\min_{v\in L}\sigma(v)<b
\]
and
\[
\max_{w\in R}\sigma(w)-\min_{w\in R}\sigma(w)<a.
\]
\end{lemma}

\begin{proof}
Let \(\tau\) be the predecessor of \(\sigma\) on the orbit, and suppose
\(r_R(\tau)\) vertices of \(R\) fire from \(\tau\).  By
\eqref{eq:confined}, whether or not \(v\in L\) fires, its chip count after
the update lies in
\[
[r_R(\tau),\,r_R(\tau)+b-1].
\]
This interval has diameter \(b-1\).  The proof on \(R\) is symmetric.
\end{proof}

\section{Two-parameter conjugates}

Order the vertices within each side so that
\[
\sigma(L_1)\ge\cdots\ge\sigma(L_a),\qquad
\sigma(R_1)\ge\cdots\ge\sigma(R_b).
\]
For \(0\le k\le a\) and \(0\le\ell\le b\), define
\(c^{k,\ell}\sigma\) by
\[
c^{k,\ell}\sigma(L_i)=
\begin{cases}
\sigma(L_i)+\ell-b,&i\le k,\\
\sigma(L_i)+\ell,&i>k,
\end{cases}
\tag{3.1}\label{eq:conjL}
\]
and
\[
c^{k,\ell}\sigma(R_j)=
\begin{cases}
\sigma(R_j)+k-a,&j\le\ell,\\
\sigma(R_j)+k,&j>\ell.
\end{cases}
\tag{3.2}\label{eq:conjR}
\]
The chip change on \(L\) is \(a\ell-bk\), while that on \(R\) is
\(bk-a\ell\).  Hence
\[
\chips{c^{k,\ell}\sigma}=\chips{\sigma}.
\tag{3.3}\label{eq:total}
\]

\begin{proposition}[Activity invariance]\label{prop:conjugacy}
If \(\sigma\) and \(c^{k,\ell}\sigma\) are nonnegative configurations,
then
\[
\activity(c^{k,\ell}\sigma)=\activity(\sigma).
\]
\end{proposition}

\begin{proof}
Write
\[
z_t(L_i)=u_t(c^{k,\ell}\sigma,L_i)-u_t(\sigma,L_i)
\]
and define \(z_t(R_j)\) analogously.  We prove by induction that
\[
\begin{aligned}
-1\le z_t(L_i)\le0&\quad(i\le k),&
0\le z_t(L_i)\le1&\quad(i>k),\\
-1\le z_t(R_j)\le0&\quad(j\le\ell),&
0\le z_t(R_j)\le1&\quad(j>\ell).
\end{aligned}
\tag{3.4}\label{eq:z-bounds}
\]
The case \(t=1\) follows from the signs of the perturbations in
\eqref{eq:conjL}--\eqref{eq:conjR}.

Assume \eqref{eq:z-bounds} at time \(t\).  Then
\[
-\ell\le\sum_{j=1}^b z_t(R_j)\le b-\ell.
\tag{3.5}\label{eq:zRsum}
\]
For \(i\le k\), the difference between the two chip counts at time \(t\)
is
\[
\ell-b+\sum_jz_t(R_j)-b z_t(L_i).
\tag{3.6}\label{eq:state-difference}
\]
If \(z_t(L_i)=0\), this is nonpositive by \eqref{eq:zRsum}; hence a firing
in the conjugate forces a firing in the original.  If
\(z_t(L_i)=-1\), it is nonnegative; hence a firing in the original forces
a firing in the conjugate.  Thus \(-1\le z_{t+1}(L_i)\le0\).
For \(i>k\), replace the initial term \(\ell-b\) in
\eqref{eq:state-difference} by \(\ell\); the same endpoint calculation
gives \(0\le z_{t+1}(L_i)\le1\).

The proof on \(R\) is identical, using
\[
-k\le\sum_{i=1}^a z_t(L_i)\le a-k
\]
and the perturbations \(k-a\) and \(k\).  This proves
\eqref{eq:z-bounds} for all \(t\).  The total difference between the two
cumulative firing counts is therefore uniformly bounded by \(a+b\), so
their activities agree.
\end{proof}

\section{A nonnegative coverage witness}

Put
\[
y_i=\sigma(L_i),\qquad x_j=\sigma(R_j).
\]
For \(1\le\ell\le b\), define
\[
k_\ell=\#\{i:y_i+\ell\ge b\},\qquad
h_\ell=\#\{i:y_i+\ell\ge2b\}.
\tag{4.1}\label{eq:kh}
\]

\begin{lemma}[Ferrers count]\label{lem:ferrers}
If \(\sigma\) satisfies \eqref{eq:confined}, then
\[
\sum_{\ell=1}^b(x_\ell+k_\ell+h_\ell)=\chips{\sigma}+a.
\]
\end{lemma}

\begin{proof}
Clearly \(\sum_\ell x_\ell=\chips{\sigma_R}\).  A fixed
\(y_i\in\{0,\ldots,2b-1\}\) contributes
\[
\min(b,y_i+1)
\]
to \(\sum_\ell k_\ell\), and
\[
\max(0,y_i-b+1)
\]
to \(\sum_\ell h_\ell\).  Their sum is \(y_i+1\).  Summing over \(i\)
gives the result.
\end{proof}

\begin{proposition}[Coverage witness]\label{prop:witness}
Suppose \(\sigma\) lies on a nontrivial periodic orbit and
\[
\chips{\sigma}>2ab-a-b.
\]
Then some nonnegative conjugate \(\tau=c^{k,\ell}\sigma\) has every
vertex of \(R\) fire in round \(0\) or round \(1\).
\end{proposition}

\begin{proof}
By Lemma~\ref{lem:ferrers},
\[
\sum_{\ell=1}^b(x_\ell+k_\ell+h_\ell)
=\chips{\sigma}+a>b(2a-1).
\]
All summands are integers, so for some \(\ell\),
\[
x_\ell+k_\ell+h_\ell\ge2a.
\tag{4.2}\label{eq:large-summand}
\]
Set \(k=k_\ell\) and \(\tau=c^{k,\ell}\sigma\).

By the definition of \(k_\ell\), subtracting \(b\) from the first \(k\)
shifted \(L\)-coordinates leaves them nonnegative; the other shifted
coordinates are also nonnegative.  Since \(h_\ell\le a\),
\eqref{eq:large-summand} gives \(x_\ell+k\ge a\), so the first \(\ell\)
shifted \(R\)-coordinates are nonnegative as well.  Hence \(\tau\) is a
nonnegative configuration.

Lemma~\ref{lem:spread} gives \(x_\ell-x_b<a\).  It follows from
\eqref{eq:conjR} that
\[
\min_{w\in R}\tau(w)=x_\ell+k-a.
\tag{4.3}\label{eq:minR}
\]
Exactly \(h_\ell\) vertices of \(L\) fire in round \(0\) from \(\tau\).
Combining \eqref{eq:large-summand} and \eqref{eq:minR},
\[
\min_{w\in R}\tau(w)+r_L(\tau)
=x_\ell+k-a+h_\ell\ge a.
\]
Every \(R\)-vertex that does not fire in round \(0\) therefore receives
enough chips to fire in round \(1\).
\end{proof}

\section{Propagation and proof of the theorem}

\begin{lemma}[Two-round wave]\label{lem:wave}
If every vertex on one side of a nonnegative configuration fires at least
once in rounds \(t,t+1\), then every vertex on the other side fires at
least once in rounds \(t+1,t+2\).
\end{lemma}

\begin{proof}
Suppose the covered side is \(R\), and fix \(v\in L\).  Assume \(v\) does
not fire in round \(t+1\).  Across rounds \(t,t+1\), it receives at least
\(b=\deg(v)\) chips.  If \(v\) fired in round \(t\), its chip count at the
start of that round was at least \(b\), offsetting the \(b\) chips sent.
If it did not fire, there is no subtraction.  In either case \(v\) has at
least \(b\) chips at the start of round \(t+2\), so it fires then.  The
other direction is symmetric.
\end{proof}

\begin{corollary}\label{cor:half}
Under the hypotheses of Proposition~\ref{prop:witness},
\(\activity(\sigma)\ge1/2\).
\end{corollary}

\begin{proof}
Let \(\tau\) be the conjugate supplied by
Proposition~\ref{prop:witness}.  Lemma~\ref{lem:wave} produces alternating
coverage windows
\[
R:[0,1],\quad L:[1,2],\quad R:[2,3],\quad L:[3,4],\ldots.
\]
Thus every vertex has asymptotic firing rate at least \(1/2\), so
\(\activity(\tau)\ge1/2\).  Proposition~\ref{prop:conjugacy} gives
\(\activity(\sigma)=\activity(\tau)\).
\end{proof}

\begin{proof}[Proof of Theorem~\ref{thm:main}]
Move \(\sigma\) to its eventual periodic orbit.  Chip total is conserved.
The strict middle interval excludes a fixed point: on a connected graph,
a fixed point has either no firing vertices or all firing vertices
\cite{bitar-goles}, which gives a chip total outside this interval.
Thus the orbit is nontrivial, and Corollary~\ref{cor:half} applies:
\[
\activity(\sigma)\ge\frac12.
\]

Define the complement
\[
\sigma_c(v)=2\deg(v)-1-\sigma(v).
\]
Complementation interchanges firing and waiting and commutes with the
update \cite[Lemma~2.3]{jiang}.  Moreover,
\[
\chips{\sigma_c}=4ab-a-b-\chips{\sigma}.
\]
The middle interval is invariant under this transformation.  Applying
Corollary~\ref{cor:half} to \(\sigma_c\) gives
\[
1-\activity(\sigma)=\activity(\sigma_c)\ge\frac12.
\]
Hence \(\activity(\sigma)=1/2\).  By
\cite[Proposition~2.5]{jiang}, all vertices fire equally often on the
periodic orbit, so each individual vertex has firing density \(1/2\).

If the period were at least \(3\), each vertex's firing word would be
non-clumpy.  Thus each such word omits cyclically adjacent \(1\)'s or
omits cyclically adjacent \(0\)'s.  At density \(1/2\), either possibility
forces that vertex's word to alternate.
After two rounds every vertex has fired exactly once, and the configuration
returns.  Thus the period is \(2\).
\end{proof}

\section*{Acknowledgement}

\small
The author acknowledges the use of AI-assisted tools in the development
and preparation of this work and takes full responsibility for its
mathematical content.

\end{document}